\documentclass[leqno,11pt]{amsart}

\usepackage{amssymb}
\usepackage{amsmath}
\usepackage{enumerate}
\usepackage[pdftex]{graphicx}
\usepackage{graphicx, color}

\def\s{\mathbb{S}}

\def\R{\mathbb{R}}

\def\Q{\mathbb{Q}}

\DeclareMathOperator{\arctanh}{arctanh} 

\DeclareMathOperator{\sech}{sech}

\DeclareMathOperator{\sn}{sn}
\DeclareMathOperator{\cn}{cn}
\DeclareMathOperator{\dn}{dn}
\DeclareMathOperator{\am}{am}

\newtheorem{theorem}{Theorem}[section]

\newtheorem{corollary}[theorem]{Corollary}

\theoremstyle{definition}

\newtheorem{remark}[theorem]{Remark}
\newtheorem{example}[theorem]{Example}

\numberwithin{equation}{section}

\begin{document}

\title[Spherical curves]{Spherical curves whose curvature \\ depends on distance to a great circle\\
\vspace{0.5cm} \normalfont{\it D\lowercase{edicated to }P\lowercase{rofessor }\'O\lowercase{scar }G\lowercase{aray, in memoriam}}}

\author[I. Castro]{Ildefonso Castro}
\address{Departamento de Matem\'{a}ticas \\
Universidad de Ja\'{e}n \\
23071 Ja\'{e}n, Spain and IMAG, Instituto de Matemáticas de la Universidad de Granada.}
\email{icastro@ujaen.es}

\author[I. Castro-Infantes]{Ildefonso Castro-Infantes}
\address{Departamento de Matem\'aticas \\
Universidad de Alicante \\
Alicante, Spain} 
\email{ildefonso.castro@ua.es}

\author[J. Castro-Infantes]{Jes\'{u}s Castro-Infantes}
\address{Departamento de Geometr\'{\i}a y Topolog\'{\i}a \\
Universidad de Granada \\
18071 Granada, Spain} 
\email{jcastroinfantes@ugr.es}




\date{}

\begin{abstract}
Motivated by a problem posed by David A. Singer in 1999 and by the elastic spherical curves, we study the spherical curves whose curvature is expressed in terms of the distance to a great circle (or from a point). By introducing the notion of {\em spherical angular momentum}, we provide new characterizations of some well known curves, like the mentioned elastic curves, spherical catenaries, loxodromic-type spherical curves, the Viviani's curve, and the spherical Archimedean spirals curves. Furthermore, we show that they may be obtained as critical points of some energy curvature functionals. We also find out several new families of spherical curves whose intrinsic equations are expressed in terms of elementary functions or Jacobi elliptic functions, and we are able to get arc length parametrizations of them.
\end{abstract}

\maketitle


\section{Introduction}

The plane curves are uniquely determined up to rigid motion by its intrinsic equation giving its curvature $\kappa $ as a function of its arc-length. However, such curves are impossible to find explicitly
in practice in most cases, due to the difficulty in solving the three quadratures appearing in the integration process. In \cite{S99}, David A.\ Singer considered a different sort of problem:
 {\em Can a plane curve be determined if its curvature is given in terms of its
position?} 

Probably, the most interesting solved
problem in this setting corresponds to the Euler elastic curves, 
whose curvature is proportional to one of the coordinate functions,
e.g.\ $ \kappa(x, y) = c\, y $.
Motivated by the above question and by the classical elasticae, the authors studied in \cite{CCI16}
the plane curves whose curvature depends on the distance to a line (say the $x$-axis, and so $\kappa=\kappa(y)$) and in \cite{CCIs17}
the plane curves whose curvature depends on the distance from a point (say the origin, and so $\kappa=\kappa(r)$, $r =\sqrt{x^2+y^2}$) requiring in both cases the computation of three quadratures too.
But the simple case $\kappa (r)=r$, where elliptic integrals appear, illustrated that the
fact that the corresponding differential equation is integrable by quadratures does not mean that it is easy to perform
the integrations. In \cite{S99}, only the very pleasant special case of the classical Bernoulli lemniscate, $r^2=3 \cos
2\theta$ in polar coordinates, was solved explicitly, where the corresponding elliptic integral becomes elementary.

In this paper, we pay our attention to the Singer's problem version for curves lying in a sphere: 
\begin{quote} 
{\em Can a spherical curve be determined when its curvature is given in terms of its position?}
\end{quote}
 The geodesic curvature $\kappa$ of a spherical curve $\xi $ given as a function of its arc length $s$ determines the curve (up to isometries of the sphere) by integration of its Frenet equations.
 However, it is expectable that if the curvature $\kappa $ of $\xi =(x,y,z)$ is given by a function of its position, i.e. $\kappa = \kappa (x,y,z)$, the situation becomes quite complicated since the general form of this problem is equivalent to solving the non linear differential equation
$$
\left| 
\begin{array}{ccc}
x(s) & y(s) & z(s) \\
\dot x(s) & \dot y(s) & \dot z(s) \\
\ddot x(s) & \ddot y(s) & \ddot z(s)
\end{array}
\right| 
=\kappa (x(s),y(s),z(s))
$$
with the constraints 
$$
x(s)^2+y(s)^2+z(s)^2=1 \quad\text{and}\quad\dot x(s)^2+\dot y(s)^2+\dot z(s)^2=1.
$$ 

The purpose of this article is the study of the aforementioned cases of the classical Singer's problem in the setting of spherical curves, considering geodesics in the role of lines. Concretely, we consider a curve $\xi =(x,y,z)$ lying in the unit sphere $\mathbb{S}^2 $ centred at the origin and write $z=\sin \varphi$ ($\varphi $ being the latitude of $\xi$); we aim to control those curves $\xi $ whose geodesic curvature $\kappa$ satisfies the condition $\kappa \!=\! \kappa (\varphi) \!\Leftrightarrow \! \kappa \!=\! \kappa (z)$. We point out that this condition includes both types of problems involving curvature and distance, since  $\varphi $ is the distance to the equator (the great circle $\varphi \!=\! 0 \!\Leftrightarrow  \!z\!=\!0$) and the colatitude $\pi/2 -\varphi$ is the distance to the North pole (the point  (0,0,1)).

As in the Euclidean case, it will be necessary again the computation of three quadratures when $\kappa (z)$ is a continuous function (see Theorem \ref{quadratures})  and the key tool will be the notion of {\em spherical angular momentum}, which completely determines a spherical curve (up to a family of distinguished isometries) in relation with its relative position with respect to a fixed geodesic.

In this way, we find out several interesting new families of spherical curves whose intrinsic equations can be expressed in terms of elementary or Jacobi elliptic functions. 
We also provide new characterizations of some well known curves, like elastic-type spherical curves, spherical catenaries, loxodromic-type spherical curves, the Viviani's curve, and the spherical Archimedean spirals curves. In addition, we show that they may be obtained as critical points of some energy curvature functionals. 

In a forthcoming paper \cite{CCI21} we have studied helicoidal minimal surfaces in $\mathbb{S}^3$ by considering surfaces  that are invariant under a helicoidal motion in the $3$-sphere, that is, the composition of two independent rotations in $\mathbb{S}^3$.



%


\section{Spherical curves such that $\kappa = \kappa (z)$ and the spherical angular momentum}\label{Sect2}
We introduce a smooth function associated to any spherical curve, which completely determines it (up to a family of distinguished isometries) in relation with its relative position with respect to a fixed geodesic.

Indeed, let  $\xi=\xi(s)\colon I\subseteq \R \rightarrow \s^2$ be an immersed curve parametrized by the arc length,
i.e.\ $|\xi (s)|=|\dot \xi (s)|=1$, for any $s\in I$, where $I$ is some interval in $\R$. 
Along the paper, $\dot \,$ will denote derivative with respect to $s$ and $\langle
\cdot, \cdot \rangle $ and $\times$ the Euclidean inner product and the cross product in $\R^3$ respectively. Let $T=\dot \xi $ be the unit tangent vector and $N=\xi \times \dot
\xi$ the unit normal vector of $\xi$. If $\nabla$ is the connection in $\s^2$, the oriented geodesic curvature $\kappa$ of $\xi$
is given by the Frenet equation $\nabla_T T = \kappa N$, which implies
\begin{equation}\label{Frenetbis}
 \ddot \xi =-\xi + \kappa N, \quad \dot N=-\kappa \,  \dot \xi
\end{equation}
and so $\kappa = \det (\xi, \dot \xi, \ddot \xi)$.

We are interested in the geometric condition that the curvature of $\xi$ depends on the distance to a geodesic of $\s^2$. If ${\bf e}\in \R^3 $ is a unit length vector, then $\langle \xi, {\bf e} \rangle$ is the signed distance to the orthogonal plane to $\bf e$ passing through the origin.  Without restriction, we consider ${\bf e}:=(0,0,1)$ and write $\xi=(x,y,z)$ with $x^2+y^2+z^2=1$. 
So we can pay our attention to study the condition $\kappa=\kappa(z)$ since $z=\langle \xi, {\bf e} \rangle$ represents the signed distance to the great circle $\s^2 \cap \{ z=0 \}$.
Concretely, we use geographical coordinates in $\s^2$ and write 
$$\xi=(\cos \varphi  \cos \lambda , \cos \varphi  \sin \lambda , \sin \varphi ), \ -\pi/2 \leq \varphi \leq \pi/2, \, -\pi < \lambda \leq \pi .$$ 
Then it is interesting to notice that the latitude $\varphi$ is the signed distance to the equator $\s^2 \cap \{ z=0 \}\equiv \varphi =0$ and, in addition, the colatitude $\pi/2-\varphi$ gives the distance to the North pole $(0,0,1)$.

At a given point $\xi (s)$ on the curve, we introduce the {\em spherical angular momentum} (with respect to the $z$-axis) $\mathcal K (s)$ as the (signed) volume of the parallelepiped spanned by the position $\xi (s)$, the unit tangent $T(s)$ and the vector ${\bf e}:=(0,0,1)$. Concretely, we define
\begin{equation}\label{spherical momentum}
\mathcal K(s) := - \det (\xi(s),T(s), {\bf e})=-\langle N(s), {\bf e} \rangle  = \dot x(s) y(s) -x(s) \dot y(s).
\end{equation}
In physical terms, as a consequence of Noether's Theorem (cf.\ \cite{A78}), $\mathcal K$ may be described as the angular momentum of a particle of unit mass with unit speed and spherical trajectory $\xi (s)$. We point out that $\mathcal K$ assumes values in $[-1,1]$ and it is well defined, up to the sign, depending on the orientation of the normal to $\xi$. 
In geographical coordinates, $\mathcal K$ is given by
\begin{equation}\label{K momentum}
\mathcal K = - \dot \lambda  \cos^2 \varphi .
\end{equation}

The unit-speed condition on $\xi$ implies that $\dot \varphi^2  + \dot \lambda^2  \cos^2 \varphi = 1$ and, assuming $\varphi$ is non constant and using \eqref{K momentum}, we deduce that
\begin{equation}\label{dif}
ds=\frac{d\varphi}{\sqrt{1-\dot \lambda ^2  \cos^2 \varphi}}=\frac{\cos \varphi \, d\varphi}{\sqrt{\cos^2 \varphi-\mathcal K^2}}=\frac{dz}{\sqrt{1-z^2-\mathcal K^2}}
\end{equation}
and 
\begin{equation}\label{long}
d\lambda= - \frac{\mathcal K ds}{\cos^2 \varphi}=  \frac{\mathcal K ds}{z^2-1} .
\end{equation}
Hence, given $\mathcal K=\mathcal K (z)$ as an explicit function, looking at \eqref{dif} and \eqref{long}, one may attempt to compute $z(s)$ (and so $\varphi (s)$) and $\lambda (s) $ in three steps: integrate \eqref{dif} to get $s=s(z)$, invert to get $z=z(s)$ and integrate \eqref{long} to get $\lambda =\lambda (s)$. We remark that the integration constants appearing in \eqref{dif} and \eqref{long} simply mean a translation of the arc parameter and a rotation around the $z$-axis respectively. 

In addition, using \eqref{Frenetbis} and \eqref{spherical momentum}, we have that 
$\dot {\mathcal  K} = - \langle \dot N, {\bf e} \rangle = \kappa \langle \dot \xi, {\bf e} \rangle = \kappa \dot z $ and, if we take into account the assumption $\kappa=\kappa(z)$ (being $z$ non constant), we finally arrive at
\begin{equation}\label{anti K}
d\mathcal K= \kappa (z) dz,
\end{equation}
that is,  $\mathcal K (z)$ can be interpreted as an anti-derivative of $\kappa (z)$.

As a summary, we can determine by quadratures in a constructive explicit way the spherical curves such that $\kappa=\kappa(z)$, in the spirit of  \cite[Theorem 3.1]{S99}.
\begin{theorem}\label{quadratures}
Let $\kappa=\kappa(z)$ be a continuous function.
Then the problem of determining locally a spherical curve whose curvature is $\kappa(z)$
---$z$ representing the (non constant) signed distance to the great circle $z\!=\!0$---
with spherical angular momentum $\mathcal K (z)$ satisfying \eqref{anti K}, 
is solvable by quadratures considering the unit speed curve $\!\xi(s)\!=\!(x(s),y(s),z(s))$, with
$x(s)=\cos \varphi (s) \cos \lambda (s), \, y(s)= \cos \varphi (s) \sin \lambda (s), 
\,  z(s)= \sin \varphi (s)$, where $\varphi(s)$ and $\lambda(s)$ are obtained through \eqref{dif} and \eqref{long} after inverting $s=s(z)$. Such a curve is uniquely determined by $\mathcal K (z)$ up to a rotation around the $z$-axis (and a translation of the arc parameter $s$).
\end{theorem}

	In other words, any spherical curve $\xi=(x,y,z):I\subseteq \R \rightarrow \s^2$, with $z$ non-constant, is uniquely determined by its spherical angular momentum $\mathcal K$ as a function of its coordinate $z$, that is, by $\mathcal K= \mathcal K(z)$. The uniqueness is modulo rotations around the $z$-axis. Moreover, the curvature of $\xi $ is given by $\kappa (z)=\mathcal K' (z)$.

\begin{remark}\label{c}
If we prescribe a continuous function $\kappa=\kappa(z)$ as curvature, the proof of Theorem~\ref{quadratures} clearly implies the computation of three quadratures, following the sequence:
\begin{enumerate}[\rm (i)]
\item[\rm (i)] A one-parameter family of anti-derivatives of $\kappa (z)$:
$$
\int \! \kappa (z) dz = \mathcal K(z).
$$
\item[\rm (ii)] Arc-length parameter $s$ of $\xi =(x,y,z) $ in terms of $z$, defined ---up to translations of the parameter--- by the integral:
$$
s=s(z)=\int\!\frac{dz}{\sqrt{1-z^2-\mathcal K(z)^2}},
$$
where $-\sqrt{1-z^2}<\mathcal K(z)<\sqrt{1-z^2}$, and
inverting $s=s(z)$ to get $z=z(s)$ and so the latitude of $\xi $ is
$
\varphi(s)=\arcsin z(s)
$.
\item[\rm (iii)] Longitude of $\xi=(\cos \varphi  \cos \lambda , \cos \varphi  \sin \lambda , \sin \varphi )$ in terms of $s$, defined ---up to a rotation around the $z$-axis--- by the integral:
$$
\lambda (s)=\int \! \frac{\mathcal K(z(s))}{z(s)^2-1}ds ,
$$
where $|z(s)|<1$.
\end{enumerate}
We note that we get a one-parameter family of spherical curves satisfying $\kappa=\kappa(z)$ according to the spherical angular momentum $\mathcal K (z)$ chosen in {\rm (i)} and verifying $\mathcal K(z)^2 +z^2 < 1 $. It will distinguish geometrically the curves inside a same family by their relative position with respect to the equator (or the $z$-axis). 
\end{remark}


We now show two illustrative examples applying steps {\rm (i)-(iii)} of the algorithm described in Remark~\ref{c}:

\begin{example}[$\kappa\!\equiv\! 0$]
Then $\mathcal K \!\equiv \! c \in \R$, and $s\!=\!\int \frac{dz}{ \sqrt{1-c^2-z^2}}=\arcsin \frac{z}{\sqrt{1-c^2}} $, with $|c|<1$.
Therefore $z(s)=\sqrt{1-c^2} \sin s $. This gives that $\lambda(s)=-\arctan (c \tan s)$ and finally $\xi (s)=(\cos s,-c \sin s, \sqrt{1-c^2} \sin s) $.  It corresponds to the great circle $\s^2 \cap \{ \sqrt{1-c^2}\,y+c\,z=0 \}$.
Up to rotations around the $z$-axis, they provide arbitrary great circles in $\s^2$, except the equator. As a consequence of Theorem~\ref{quadratures}, we deduce that the great circle $\s^2 \cap \{ \sqrt{1-c^2}\,y+c\,z=0 \}$ is the only spherical curve (up to rotations around the $z$-axis) with constant spherical angular momentum $\mathcal K\!\equiv\! c$ (see Figure~\ref{GreatCircles}). 

\begin{figure}[h!]
\begin{center}
\includegraphics[height=3cm]{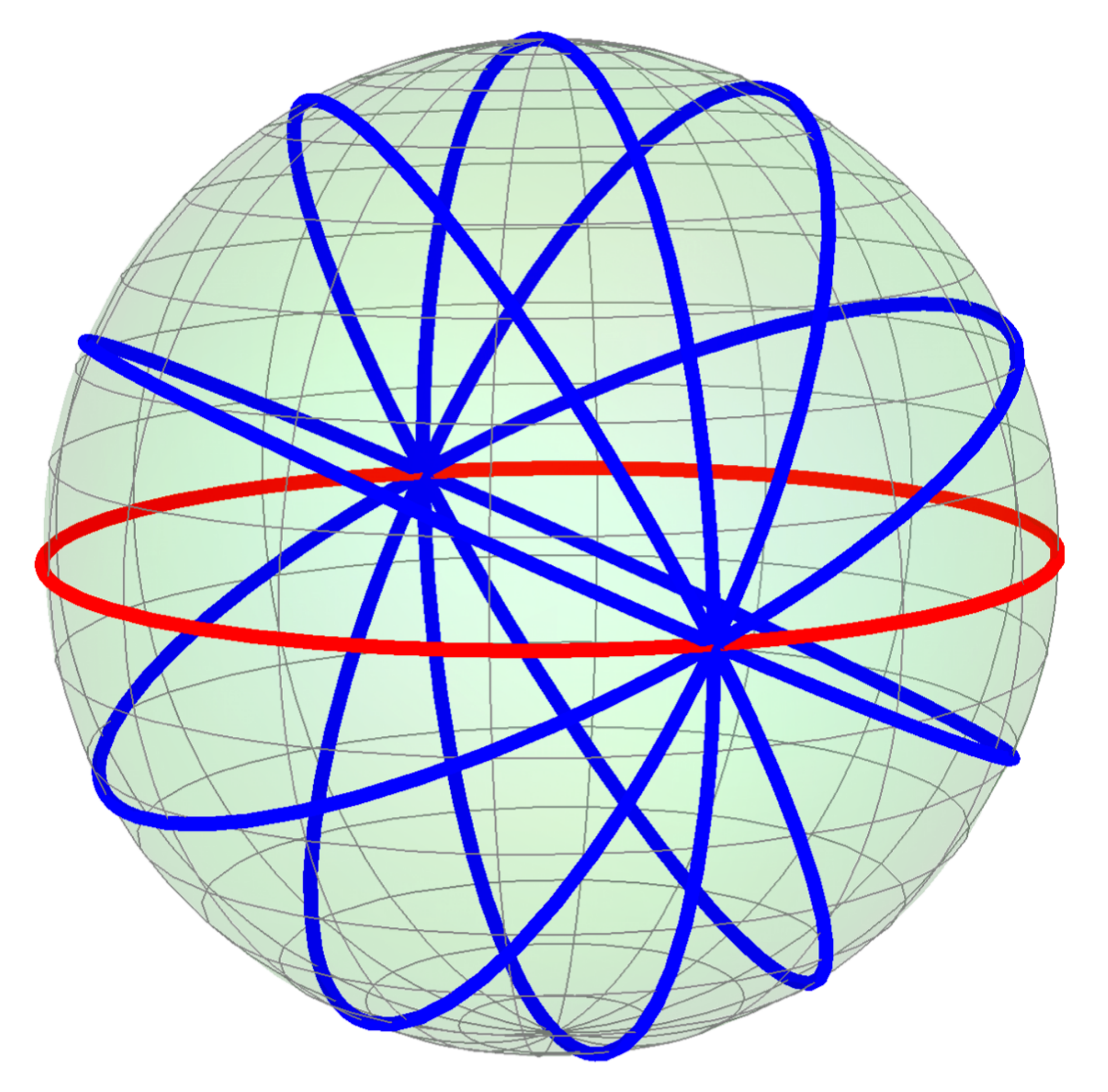}
\caption{Great circles $\s^2 \cap \{ \sqrt{1-c^2}\,y+c\,z=0 \}$: $\mathcal K\equiv c\in (-1,1)$.}
\label{GreatCircles}
\end{center}
\end{figure}
For example, taking $c=0$, we get the meridian $\s^2 \cap \{ y=0 \}$ and hence the meridians are the only spherical curves with null spherical angular momentum. We notice that the limiting cases $c=\pm 1$ lead to the equator $\s^2 \cap \{ z=0 \}$.
\end{example}
\begin{example}[$\kappa \equiv k_0 >0 $]
Now $\mathcal K (z)= k_0 z+c$, $c\in \R$. In this case, it is not difficult to get that 
$$z(s)=\frac{1}{1+k_0^2}\left( \sqrt{1-c^2+k_0^2}\, \sin \left(\sqrt{1+k_0^2}\, s\right)-c \, k_0\right)$$
with $|c|<\sqrt{1+k_0^2}$. 
But the expression of $\lambda$ is far from trivial and depends on the values of $c$. After a long computation, we deduce:
\begin{itemize}
	\item If $|c|\neq k_0$:
	$\lambda(s)=\arctan \left(\frac{\sqrt{1 - c^2 + k_0^2} + (1 - c k_0 + k_0^2) \tan({\frac12 \sqrt{1 + k_0^2} s})}{(k_0-c)\sqrt{1 + k_0^2}}\right)+$
	
	\hspace{2.5cm}$+\arctan\left(\frac{\sqrt{1 - c^2 + k_0^2} + (1 + c k_0 + k_0^2) \tan({\frac12 \sqrt{1 + k_0^2} s})}{(k_0+c)\sqrt{1 + k_0^2}}\right).$
	\item If $c=k_0$:
	$\lambda(s)= \arctan\left( \frac{1 - (1 + 2 k_0^2) \tan(\frac12 \sqrt{1 + k_0^2} s)}{2 k_0 \sqrt{1 + k_0^2}}\right).$
	\item If $c=-k_0$:
	$\lambda(s)= \arctan\left( \frac{1 + (1 + 2 k_0^2) \tan(\frac12 \sqrt{1 + k_0^2} s)}{2 k_0 \sqrt{1 + k_0^2}}\right).$
\end{itemize}
Of course, up to rotations around the $z$-axis, we get all the non-parallel small circles of $\s^2$.
The parameter $c$ distinguishes the position of the circle with respect to the equator.
If $0\leq |c|<1$ the circles intersect the equator transversely; in particular, when $c=0$ we obtain the orthogonal circles to the equator. 
If $c=\pm 1$, the circles are tangent to the equator.
Finally, if $1<|c|<\sqrt{1+k_0^2}$, the circles do not intersect the equator (see Figure~\ref{SmallCircles}).

\begin{figure}[h!]
	\begin{center}
		\includegraphics[height=3cm]{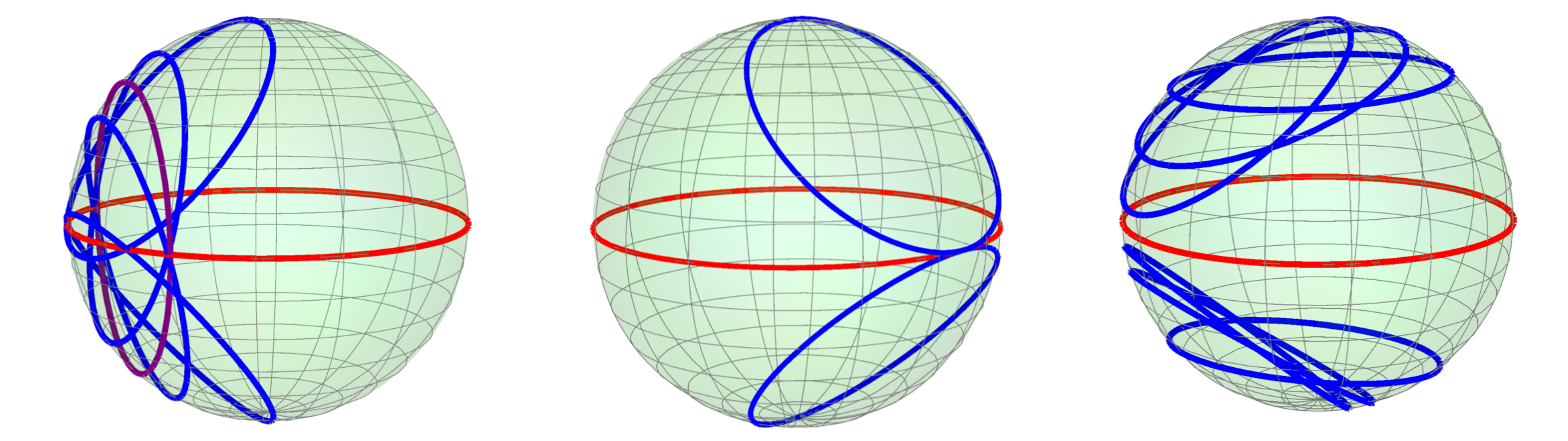}
		\caption{ 
			Small circles: $\mathcal K (z)= k_0 \, z + c, \ k_0 >0$; \newline $0\leq |c|<1$ (left), $c=\pm 1$ (center), $1<|c|<\sqrt{1+k_0^2}$ (right).}
		\label{SmallCircles}
	\end{center}
\end{figure}

For example, if $c=0$  we arrive (up to rotations around the $z$-axis) at $\lambda (s)=\arctan \left( \cos (\sqrt{1+k_0^2} \, s)/k_0\right) $ and finally 
$$ \xi (s)=\frac{1}{\sqrt{1+k_0^2}}\left(k_0,\cos \left(\sqrt{1+k_0^2}\, s\right),\sin \left(\sqrt{1+k_0^2}\, s\right)\right),  $$
that corresponds to the small circle $\s^2 \cap \{ x=k_0/\sqrt{1+k_0^2} \}$. This is the only spherical curve, up to rotations around the $z$-axis, with spherical angular momentum $\mathcal K (z)= k_0 z$
(see Figure~\ref{SmallCircles}). 
\end{example}
\begin{remark}\label{difficulties}
{\rm The main difficulties one can find carrying on the strategy described in Remark~\ref{c}  (or in Theorem~\ref{quadratures}) to determine a spherical curve whose curvature is $\kappa=\kappa(z)$ are the following:
\begin{enumerate}[\rm (1)]
\item The integration of $s=s(z)$:  Even in the case that $\mathcal K (z)$ were polynomial,
the corresponding integral is not necessarily elementary. For example, when $\mathcal K (z)$ is a quadratic
polynomial, it can be solved using Jacobian elliptic functions (see e.g.\ \cite{BF71}). This is
equivalent to $\kappa (z)$ be linear, i.e.\ $\kappa (z)= 2a z+ b$, $a,b \in \R$. We will study such
spherical curves in Section 3.
\item The previous integration gives us $s=s(z)$; it is not always possible to obtain explicitly $z=z(s)$, what is necessary to determine the curve.
\item Even knowing explicitly $z=z(s)$, the integration to get $\lambda(s)$ may be impossible to perform using elementary or known functions.
\end{enumerate}
}
\end{remark}
Nevertheless, along the paper we will study different families where we are successful with the procedure described in  Remark~\ref{c} and we will recover some known curves and find out new spherical curves characterized by their spherical angular momentum.


\section{Elastic-type curves on the sphere}
\subsection{A new characterization and a generalization of elastic curves}
A unit speed spherical curve $\xi$ is said to be an {\em elastica under tension $\sigma$} (see \cite{LS84}) if its curvature $\kappa $ satisfies the differential equation
\begin{equation}\label{eq_elastica}
2 \ddot \kappa + \kappa^3 + (2-\sigma) \kappa =0
\end{equation}
for some $\sigma \in \R$. They are the critical points of the elastic energy functional 
$$ \int_\xi (\kappa^2 + \sigma )ds $$ acting on spherical curves with suitable boundary conditions. If $\sigma =0$, then the constraint on arc length is removed and $\xi$ is called a {\em free elastica}. 

A possible generalization of free elasticae was considered in \cite{AGM06}, where the authors studied the elastic curves in $\s^2$ which are circular at rest. They are called {\em $\lambda$-elastic curves}. These curves are critical points of the functional 
$$ \int_\xi (\kappa + \lambda)^2 ds, \, \lambda \in \R$$ 
and are characterized by the Euler-Lagrange equation
\begin{equation}\label{eq_l-elastica}
2 \ddot \kappa + \kappa^3 + (2-\lambda^2) \kappa +2\lambda =0.
\end{equation}
It is obvious that the 0-elastic curves are the free elasticae.
The main result of this section deals with the spherical curves which are the critical points of the bending energy for variations with constant length (including both previous types of elastic curves) relating them with the case commented in part (1) of Remark~\ref{difficulties}.
\begin{theorem}\label{general elasticae}
Let $\xi $ be a spherical curve whose curvature $\kappa$ satisfies
\begin{equation}\label{condition1}
\kappa = 2a \langle \xi, {\bf e} \rangle +b , \ a\neq 0, \,  b\in \R,
\end{equation}
for some ${\bf e}\in \R^3, \ |{\bf e}|=1$. Then:
\begin{enumerate}[\rm (i)]
\item[\rm (i)] the spherical angular momentum $\mathcal K$ of $\xi $ is given by 
$$ \mathcal K =   a \langle \xi, {\bf e} \rangle ^2 + b \langle \xi, {\bf e} \rangle + c,  $$
for some $c \in \R$.
\item[\rm (ii)] $\xi $ is a critical point of the functional 
$$ \int_\xi (\kappa^2 -2b \kappa + b^2-4ac)\,ds $$ and so $\kappa $ satisfies the corresponding Euler-Lagrange equation
\begin{equation}\label{eq_generalelastica}
2 \ddot \kappa + \kappa^3 + \left(2-(b^2-4ac)\right) \kappa - 2b =0.
\end{equation}
\end{enumerate}
If $b=0$, $\xi $ is an elastica under tension $\sigma =-4ac$;
and if $c=0$, $\xi $ is a $\lambda$-elastic curve with $\lambda=-b$.

Conversely, if $\xi $ is a critical point of the functional 
$${\mathcal F}_\sigma^\lambda (\xi) :=  \int_\xi \left((\kappa + \lambda)^2 + \sigma \right) \,ds , \ \lambda, \sigma \in \R ,$$ 
then the curvature of $\xi $ can be written as in \eqref{condition1}.
\end{theorem}

\begin{remark}
It is remarkable the similitude of condition \eqref{condition1} characterizing the generalized elasticae considered in Theorem~\ref{general elasticae} with the geometric property satisfied by the classical Euler elastic curves in the plane: their curvature is proportional to one of the coordinate functions, say $\kappa (x,y) = 2\lambda y + \mu$, $\lambda\neq 0, \, \mu\in \R$ (see Section 3 in \cite{CCI16} and Section 1 in \cite{S99}).
Even something similar happens to spacelike and timelike elastic curves in Lorentz-Minkowski plane (see \cite{CCIs18a} and \cite{CCIs18b}).
\end{remark}

\begin{proof}[Proof of Theorem \ref{general elasticae}]
From \eqref{spherical momentum}, \eqref{Frenetbis} and \eqref{condition1}, we get that $$ \frac{d}{ds} \left( -\mathcal K  + a \langle \xi, {\bf e} \rangle ^2 + b \langle \xi, {\bf e} \rangle \right) =0. $$ This proves part {\rm (i)}.

We also have from \eqref{condition1}, \eqref{Frenetbis} and \eqref{spherical momentum}, that $\dot \kappa = 2a \langle \dot \xi, {\bf e} \rangle$ and $\ddot \kappa = 2a (-\langle \xi, {\bf e} \rangle - \kappa\mathcal K )$. Now we can easily check that $\kappa$ given by \eqref{condition1} satisfies \eqref{eq_generalelastica} since,
after a straightforward computation using (i) and putting $\langle \xi, {\bf e} \rangle =(\kappa -b)/2a$, we arrive at \eqref{eq_generalelastica}.
We observe that if $b=0$ then $\xi $ satisfies \eqref{eq_elastica} with $\sigma=-4ac$ and if $c=0$ then $\xi $ satisfies \eqref{eq_l-elastica} with $\lambda=-b$.
Moreover, in \cite{AGM03} it is shown that for a given differentiable function $P(\kappa)$, the critical points of the functional $\int_\xi P(\kappa)ds$ are characterized by the Euler-Lagrange equation
\begin{equation}\label{EL}
(\kappa^2 +1) P'(\kappa) +\frac{d^2 (P'(\kappa))}{ds^2}=\kappa P(\kappa).
\end{equation}
It is an exercise to check that putting $P(\kappa)=\kappa^2-2b \kappa + b^2-4ac$ in \eqref{EL} we obtain \eqref{eq_generalelastica}. This finishes the proof of part {\rm (ii)}. 

Multiplying \eqref{eq_generalelastica} by $ \dot k$ we obtain a first integral 
$$ \dot \kappa ^2 + \kappa ^4 /4+ \left(1-(b^2-4ac)/2\right) \kappa ^2 -2 b \kappa = E, $$ where $E$ is a real constant.
After a long computation involving \eqref{condition1}, part {\rm (i)}, \eqref{spherical momentum} and using that $\langle \dot \xi, {\bf e} \rangle ^2 =1-\langle
\xi, {\bf e} \rangle ^2 - \langle N, {\bf e} \rangle ^2$, since $\{\xi,\dot \xi,N\}$ is an orthonormal basis in $\R^3$, we get that
\begin{equation}\label{E}
E=4a^2 -b^2-(b^2-4ac)^2/4.
\end{equation}

On the other hand, suppose now that $\xi $ is a critical point of ${\mathcal F}_\sigma^\lambda $. 
Taking into account \eqref{EL} we deduce that $\kappa$ verifies the differential equation
\begin{equation}\label{eq_generalelasticabis}
2 \ddot \kappa + \kappa^3 + (2-\lambda^2-\sigma) \kappa + 2 \lambda =0.
\end{equation}
Multiplying \eqref{eq_generalelasticabis} by $\dot \kappa$ and integration allow us to deduce the energy $E\in\R$ of $\xi$:
\begin{equation}\label{energy_generalelastica}
E:=  \dot \kappa ^2 + \frac{\kappa^4}{4} + \left(1-\frac{\lambda^2+\sigma}{2}\right) \kappa^2 +2\lambda \kappa .
\end{equation}
We want to prove that $\kappa$ has an expression like in \eqref{condition1}. For this purpose, we first observe that if $\lambda =0$ then $\kappa \equiv 0$ is a trivial solution to \eqref{eq_generalelasticabis}. For example, we can take ${\bf e}\in \R^3, \ |{\bf e}|=1$ the unit normal vector orthogonal to the vectorial plane containing the corresponding great circle in $\s^2$. Now we must look for $a\neq 0$ and $b\in \R$ satisfying \eqref{condition1}. Comparing \eqref{eq_generalelastica} and \eqref{eq_generalelasticabis}, we take $b=-\lambda$ and observe that $\sigma $ must satisfy $\sigma = -4ac$, with $a \neq 0$, $c\in \R$. Using \eqref{energy_generalelastica}, we have that 
$ 4E+4\lambda^2+(\lambda^2+\sigma)^2=4  \dot \kappa ^2+4(\kappa + \lambda)^2+ \left( \kappa ^2 - (\lambda ^2+\sigma)\right)^2 >0 $.

But looking at \eqref{E}, $E$ must satisfy $E=4a^2 -b^2-(b^2-4ac)^2/4$ and, eliminating $c$, we finally arrive at 
$ 0<4E+4\lambda^2+(\lambda^2+\sigma)^2=16 a^2$, which allows us to obtain the searched value for $a$.
\end{proof}
The study of (free) elastic curves on the sphere has been considered under different approaches (see for example \cite{AGM03}, \cite{BC94}, \cite{J95}, \cite{LS84} or \cite{S08}), paying special attention to the closed ones. The closed $\lambda$-elastic spherical curves were studied in \cite{AGM06}.
All the mentioned articles are based on the study of the differential equation for the geodesic curvature of the spherical curve, being sometimes integrated directly in terms of Jacobi elliptic functions. 

In our approach of Theorem \ref{general elasticae}, we can choose ${\bf e}=(0,0,1)$ without loss of generality and so  $\langle \xi, {\bf e} \rangle =z$. In this way we arrive at the conditions
\begin{equation}\label{condition1bis}
\kappa (z) = 2a z + b, \, \mathcal K (z) = a z^2 + b z +c, \, a\neq 0, b,c \in \R,
\end{equation}
in the notation of Theorem \ref{quadratures}. Thus, using suitable coordinates in the sphere, we can conclude the following uniqueness result for the spherical elastic curves considered in literature.

\begin{corollary}\label{unique elasticae}
\begin{enumerate}[\rm (a)] 
\item The elasticae under tension $\sigma $ are the only spherical curves (up to rotations around $z$-axis) with spherical angular momentum $\mathcal K (z) = a z^2 +c, \, a\neq 0, c\in \R$, with $\sigma =-4ac$. In particular, the free elasticae are characterized by the spherical angular momentum $\mathcal K (z) = a z^2, \, a\neq 0$.
\item The $\lambda$-elastic curves are the only spherical curves (up to rotations around $z$-axis) with spherical angular momentum $\mathcal K (z) = a z^2 +b z, \, a\neq 0, b\in \R$, with $\lambda =-b$.
\end{enumerate}
\end{corollary}

\begin{remark}\label{rem:functional}
Inspired by the Langer and Singer work on the Kirchhoff elastic rod \cite{LS96},
in \cite{AGM04} it is  obtained by geometrical means the first integrals of the Euler-Lagrange equations of curvature
energy functionals $\int_\xi P(\kappa)ds$, where $P$ is a smooth function and $\kappa $ denotes the curvature of the spherical curve $\xi$ in $\s^3$.  Assuming that $P''(\kappa)\neq 0$, 
the critical points of the functional $\int_\xi P(\kappa)ds$ are characterized by a couple of differential equations
naturally related to a system of cylindrical coordinates in the three-sphere adapted to the curve $\xi$.

In the case that the curve $\xi $ lies in $\s^2$, from Section 3 of \cite{AGM04}, using concretely equation (33) with $\theta =0$ ($b=0$), then we have that $\cos^2 \psi = P'(k)^2 / a^2 $, $a\neq 0$, where
$\psi = \pi/2 -\varphi$ is the colatitude of $\xi$.

As a consequence, \textit{if a spherical curve $\xi$ is a critical point of $\int_\xi P(\kappa)ds$, then there exist geographical coordinates $(\varphi, \lambda)$ adapted to $\xi$ such that $$\sin \varphi = \delta \, P'(\kappa), \, \delta \neq 0.$$}
This result is consistent with Theorem \ref{general elasticae}, since if $P(\kappa)=(\kappa + \lambda)^2 + \sigma$, then
$\sin \varphi =z=\frac{\kappa -b}{2a}= \delta \,P'(\kappa)$, taking $\delta=\frac{1}{4a}$ and $\lambda =-b$.
\end{remark}

If one follows the strategy described in Remark \ref{c} 
 to determine the generalized elasticae satisfying \eqref{condition1bis}, it is necessary to perform the integral
\begin{equation}\label{pol}
s=s(z)=\int \frac{dz}{\sqrt{P(z)}}, \ P(z)=1-z^2-(az^2 + bz + c)^2.
\end{equation}
Since $P(z)$ is a fourth order polynomial, \eqref{pol} can be solved in terms of Jacobi elliptic functions once the nature and multiplicity of the roots of $P(z)$ are determined. After inverting $s=s(z)$ to get $z=z(s)$ we would arrive at the expressions of $\kappa=\kappa(s)=2a z(s)+b$ compatible with the ones given in \cite{AGM03}, \cite{BC94}, \cite{J95}, \cite{LS84} or \cite{S08} and \cite{AGM06}, at least when $b=0$ or $c=0$. 
Besides helices and circles, the borderline, orbitlike and wavelike elasticae appear and, in addition, a more general case according to the expressions of the geodesic curvature in terms of the arc parameter (see \cite{S08}).
In order to get the explicit expression of the generalized elasticae, we also have to compute
\begin{equation}
\label{long elas}
\lambda (s)=\int \! \frac{ a z(s)^2+b z(s) + c}{z(s)^2-1}\, ds,
\end{equation}
which, in general, leads to complicated elliptic integrals. In the next two sections we are going to illustrate the aforementioned computations in two interesting and attractive cases.

\subsection{Seiffert's spherical spirals}
The  Seiffert's spirals are defined as those spherical curves obtained when one moves along the surface of a sphere with constant speed, while maintaining a constant angular velocity with respect to a fixed diameter (cf.\ \cite{E00}). These curves are given in cylindrical coordinates $(r,\theta,z)$, $r^2+z^2=1$, by the parametric equations
\begin{equation}
\label{Seiff}
r=\sn(s,p), \  \theta =p\, s, \  z=\cn(s,p), \ (p>0),
\end{equation}
where $p$ is a positive constant and $\sn$  and $\cn$  are the elementary Jacobi elliptic functions (cf.\ \cite{BF71} for instance). Erd\"{o}s provided in \cite{E00} a derivation of the equations of this curve, as well as an analysis of its properties, including conditions for obtaining periodic orbits. When $p>1$, the spiral is located entirely in the northern hemisphere.

Now we prove that these curves are elastic curves with positive tension corresponding to the conditions $\kappa (z)=2az$, $\mathcal K (z)=a z^2-a$, i.e.\ $b=0$, $a+c=0$ in \eqref{condition1bis}. Then $\sigma = 4a^2>0$ and there is no restriction if we consider $a>0$. So \eqref{pol} can be written as
$$ \left( \frac{dz}{ds} \right)^2 = (1-z^2)\left(1-a^2+a^2 z^2 \right), \ \frac{a^2-1}{a^2}<z^2<1, $$
which implies that $z(s)=\cn (s,a)$ (cf.\ \cite{BF71} for instance) and thus $r(s)=\sn (s,a)$. Using that $a+c=0$ in \eqref{long elas} we get that $\lambda (s)=a\,s$ and so we arrive at the Seiffert's spirals (see Figure~\ref{Seiff spirals}). As a summary:
\begin{corollary}
\label{Seiffert}
The Seiffert's spirals \eqref{Seiff} are the only spherical curves (up to rotations around $z$-axis) with spherical angular momentum $\mathcal K (z) = p z^2 -p, \, p>0$.
\end{corollary}
\begin{figure}[h!]
\begin{center}
\includegraphics[height=6cm]{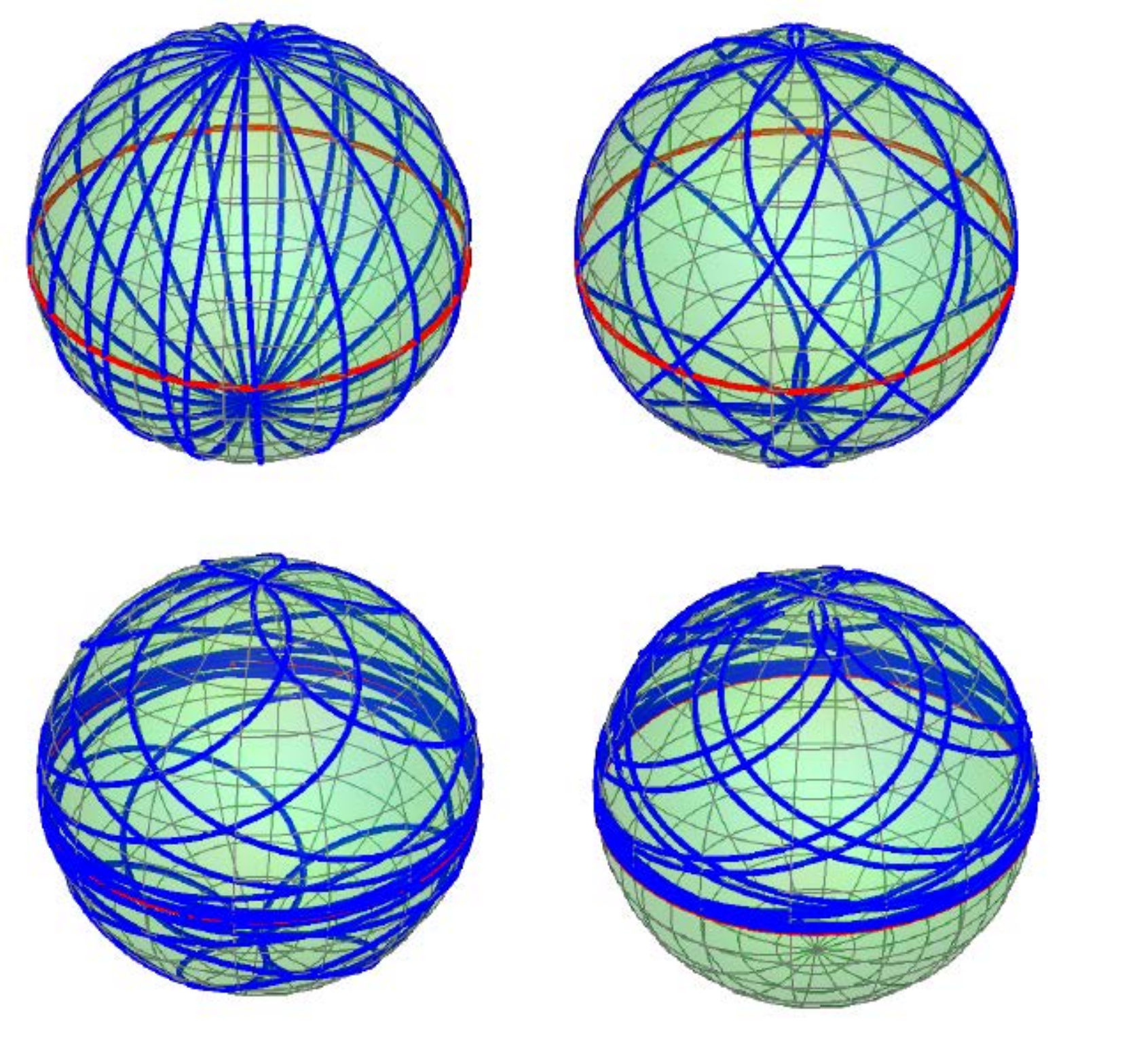}
\caption{Seiffert's spirals ($p\approx 0$, $0<p<1$, $p\approx 1$, $p>1$): $\mathcal K (z)= p z^2 -p$, $p>0$.}
\label{Seiff spirals}
\end{center}
\end{figure}

\subsection{Borderline spherical elastic curves}
We study in this section elastic curves with null energy producing elasticae under positive tension $\sigma>0$. This corresponds to the conditions $\kappa (z)=2az $, $\mathcal K (z)=a z^2-1$, i.e.\ $b=0$, $c=-1$ in \eqref{condition1bis}. Then $E=0$ according to \eqref{E} and we can take $a>0$ in order to $\sigma=4a>0$. Now \eqref{pol} leads to
$$ s=s(z)=\int \frac{dz}{z\,\sqrt{2a-1-a^2 z^2}}, \ \, 2a-1-a^2 z^2>0, $$
which implies that $a>1/2$. The above integral becomes elementary and, after inverting $s=s(z)$ and up to a translation on the parameter $s$, we get
\begin{equation}
\label{border1}
z(s)=\frac{2a-1}{a}\sech (\sqrt{2a-1}\,s).
\end{equation}
Looking at \eqref{long elas}, we get that if $a=1$ then $\lambda(s)=s$ (and $z(s)=\sech s$), and when $a>1/2$ with $a\neq 1$ we obtain 
\begin{equation}
\label{border2}
 \lambda(s)=s+ \arctan \left( \frac{\sqrt{2a-1}}{1-a}\tanh  (\sqrt{2a-1}s) \right).
\end{equation}
This family corresponds to the ``borderline elasticae'' described in \cite{S08} which are asymptotic to the equator. We show some pictures of them in Figure \ref{Borderline pics}.
In conclusion, we deduce the following uniqueness result:
\begin{corollary}
\label{borderline}
The borderline elasticae given by \eqref{border1} and \eqref{border2} are the only spherical curves (up to rotations around $z$-axis) with spherical angular momentum $\mathcal K (z) = a z^2 -1, \, a>1/2$.
\end{corollary}
We remark that when $a=1$ we recover the Seiffert's spiral corresponding to $p=1$, since $\cn(s,1)=\sech s$ (see \cite{BF71}).

\begin{figure}[h!]
\begin{center}
\includegraphics[height=3cm]{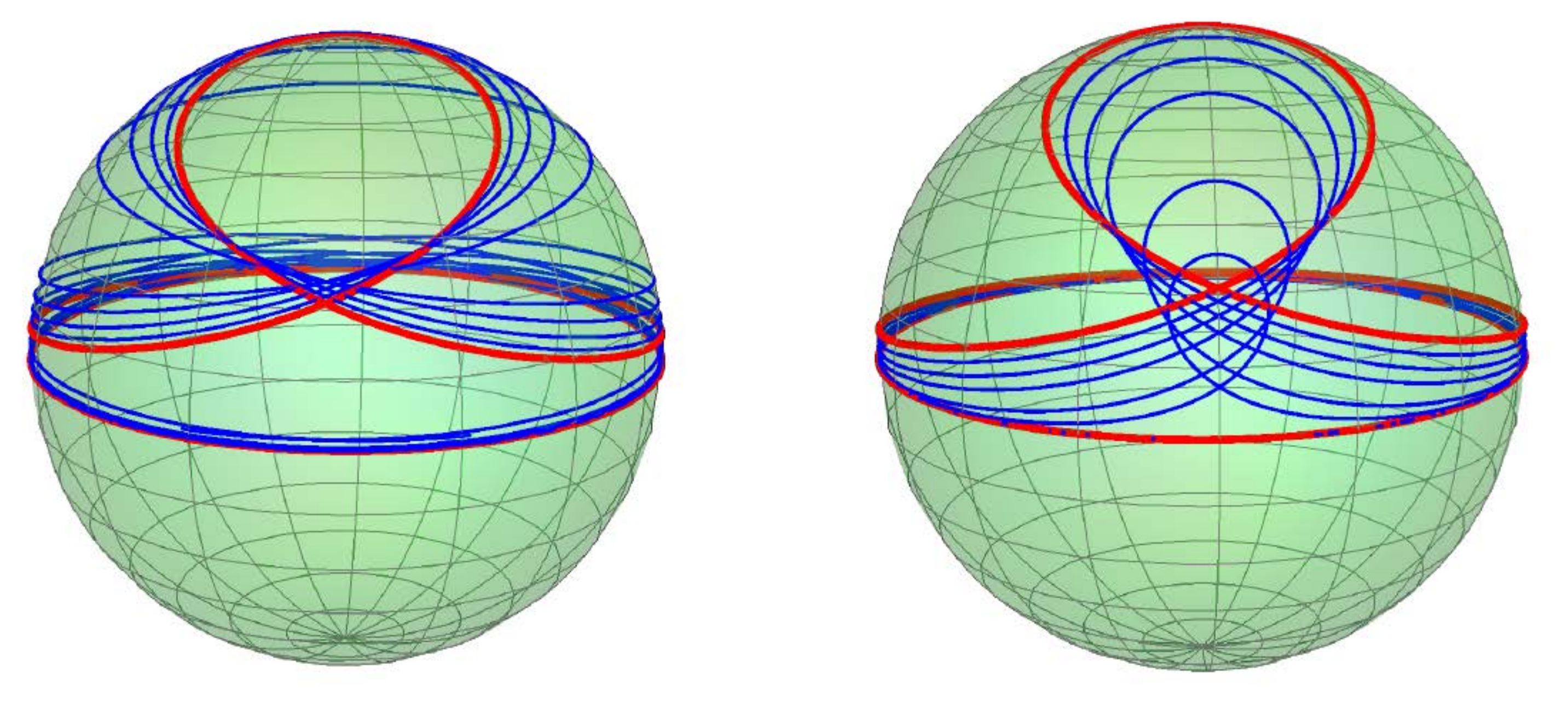}
\caption{Borderline elastic curves: $\mathcal K (z) = a z^2 -1, \, a>1/2$ (left: $1/2<a \leq 1$; right: $a\geq 1$).}
\label{Borderline pics}
\end{center}
\end{figure}

\section{Loxodromic-type curves on the sphere}

In this section we are interested in critical points of the functional
$$ \mathcal F^\mu (\xi):=\int_\xi \sqrt{\kappa^2+\mu^2}\, ds, \, \mu >0. $$
This functional was considered in Section 6 of \cite{AGM03} acting on the space of immersed closed curves in $\s^2$, motivated by total $\R^3$-curvature type functionals. We aim to connect the critical points of $\mathcal F^\mu$ with the classical loxodromes curves in $\s^2$ and others spherical curves with similar characteristics.

Taking into account Remark \ref{rem:functional}, we are devoted to study the spherical curves with curvature 
\begin{equation}
\label{k mu delta}
\kappa (z) = \frac{\mu \, z}{\sqrt{\delta^2 -z^2}}, \ \mu >0, \delta \neq 0. 
\end{equation}
We will distinguish cases according to $0 <\mu <1 $, $\mu = 1$ and $\mu >1$.

\subsection{Case $0 <\mu <1 $: spherical loxodromes}
The loxodromes are interesting curves in the sphere studied, among others, by Pedro Nunes in 1537, Simon Stevin in 1608 or Maupertuis in 1744. They are also known as rhumb lines because they make a constant angle $\alpha \in (0,\pi/2)$ with the meridians (cf.\ \cite{F93}). Analytically, using geographical coordinates $(\varphi, \lambda)$, they are defined by the equation 
\begin{equation}\label{eq_loxodrome}
d\lambda = \cot \alpha \frac{d\varphi}{\cos \varphi}.
\end{equation}
The aim of this section is the study of the spherical curves satisfying
\begin{equation}\label{k_loxodrome bis}
\kappa(\varphi)= a \tan \varphi, \ 0<a<1,
\end{equation}
or, equivalently,
\begin{equation}\label{k_loxodrome}
\kappa(z)=\frac{a z }{\sqrt{1-z^2}}, \ 0<a<1.
\end{equation}
So they correspond in \eqref{k mu delta} to the election $\mu=a\in (0,1)$ and $\delta =1$.

The trivial solution of \eqref{k_loxodrome} is given by the equator $z=0$. We follow the method described in Theorem \ref{quadratures} and Remark \ref{c}, considering the spherical angular momentum
$$ \mathcal K (z)=-a\sqrt{1-z^2}, \ 0<a<1. $$
Then we have:
$$ s\!=\!\int \frac{dz}{\sqrt{1-z^2-(-a\sqrt{1-z^2})^2}} 
\!=\!\int \frac{dz}{\sqrt{(1-a^2)(1-z^2)}}\!=\!\frac{1}{\sqrt{1-a^2}}\arcsin z, $$
and so $z(s)=\sin(\sqrt{1-a^2} \,s) $ and therefore
\begin{equation}
\label{lat lox}
\varphi(s)=\sqrt{1-a^2}\, s.
\end{equation}
On the other hand, from \eqref{lat lox}, we get:
\begin{equation}
\label{long lox}
\lambda (s)\!=\!\int \frac{a \,ds}{\cos \varphi (s)}\!=\! \frac{a}{\sqrt{1\!-\!a^2}}\log \left( \sec (\sqrt{1\!-\!a^2}\, s)+\tan (\sqrt{1\!-\!a^2}\, s)\right),
\end{equation}
where $|s|<\dfrac{\pi}{2\sqrt{1-a^2}}$.

We deduce from \eqref{lat lox} and \eqref{long lox} that $$d\lambda = \frac{a}{\sqrt{1-a^2}}\, \frac{d\varphi}{\cos \varphi}$$ and, taking into account \eqref{eq_loxodrome} and that $\mathcal K (\varphi) =-a \cos \varphi$,
we conclude the following characterization of the loxodromes.
\begin{corollary}
The loxodromes $d\lambda = \cot \alpha\, d\varphi / \cos \varphi$ are the only spherical curves (up to rotations around $z$-axis) with spherical angular momentum $\mathcal K(\varphi)=-\cos \alpha \cos\varphi $, $\alpha \in (0,\pi/2)$.
\end{corollary}
From \eqref{k_loxodrome bis} and \eqref{lat lox}, we arrive at the intrinsic equation of the loxodromes 
(see Figure \ref{loxodromes}), given by
$$
\kappa (s)= \cos \alpha \, \tan (\sin \alpha \, s), \, \alpha \in (0,\pi/2).
$$
\begin{figure}[h!]
\begin{center}
\includegraphics[height=3cm]{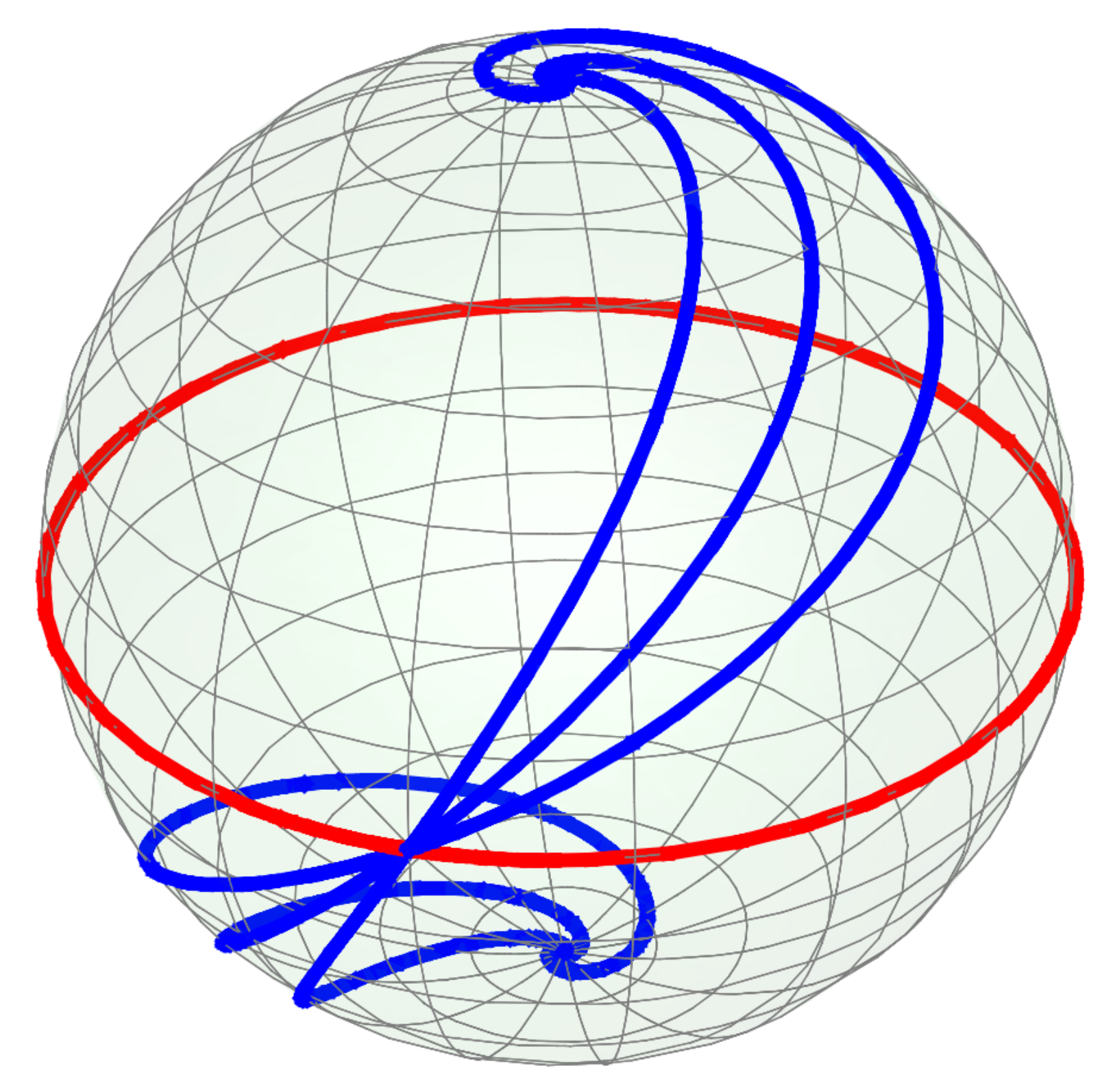}
\caption{Loxodromes: $\mathcal K(\varphi)=-\cos \alpha \cos\varphi $, $\alpha \in (0,\pi/2)$.}
\label{loxodromes}
\end{center}
\end{figure}

\subsection{Case $\mu =1$}
We now want to determine the spherical curves whose curvature is given by
\begin{equation}
\label{explicit1}
\kappa(z)=\frac{z}{\sqrt{a-z^2}}, \ 0<a<1.
\end{equation}
Looking at \eqref{k mu delta}, we are now considering $\mu=1$ and $\delta =\sqrt a$.

The trivial solution is given by the equator $z=0$.
We follow the method described in Theorem~\ref{quadratures} or Remark \ref{c} for the non trivial case, considering the spherical angular momentum $$\mathcal K (z)=-\sqrt{a-z^2}.$$ In this way, we get:
$$
s=s(z)=\int \frac{dz}{\sqrt{1-z^2-(-\sqrt{a-z^2})^2}}=\int \frac{dz}{\sqrt{1-a}}=\frac{z}{\sqrt{1-a}},
$$
and thus $z(s)=\sqrt{1-a} \, s$.

We write $a=\sin^2 \alpha$ , $0<\alpha <\pi/2$, and abbreviate
$ c_\alpha=\cos \alpha, s_\alpha =\sin \alpha$. Hence we have:
$$
\lambda (s)=\int \frac{\sqrt{s_\alpha^2-c_\alpha^2 s^2}}{1-c_\alpha^2 s^2}\, ds,
$$
which gives
\begin{equation*}
\begin{array}{c}
\lambda (s)=\dfrac{1}{c_\alpha} \arctan \left( \dfrac{c_\alpha s}{\sqrt{s_\alpha^2-c_\alpha^2 s^2}} \right) \\
 -\dfrac{1}{2} \arctan \left( \dfrac{c_\alpha s+s_\alpha^2}{c_\alpha\sqrt{s_\alpha^2-c_\alpha^2 s^2}} \right)
-\dfrac{1}{2} \arctan \left( \dfrac{c_\alpha s-s_\alpha^2}{c_\alpha\sqrt{s_\alpha^2-c_\alpha^2 s^2}} \right),
\end{array}
\end{equation*}
where $|s|<\tan \alpha$.

We observe that $|z(s)|<\sin \alpha$. Using \eqref{explicit1} and $z(s)=c_\alpha s$, we get the intrinsic
equation
\[
\kappa (s)= \dfrac{c_\alpha s}{\sqrt{s_\alpha^2-c_\alpha^2 s^2}}, \ |s|<\tan \alpha,
\]
of this family of spherical curves of loxodromic-type (see Figure \ref{zSqrtaminusz2}) characterized by the geometric angular momentum $\mathcal K (z)=-\sqrt{\sin^2\alpha-z^2}$, $0<\alpha <\pi/2$.
\begin{figure}[h!]
\begin{center}
\includegraphics[height=3cm]{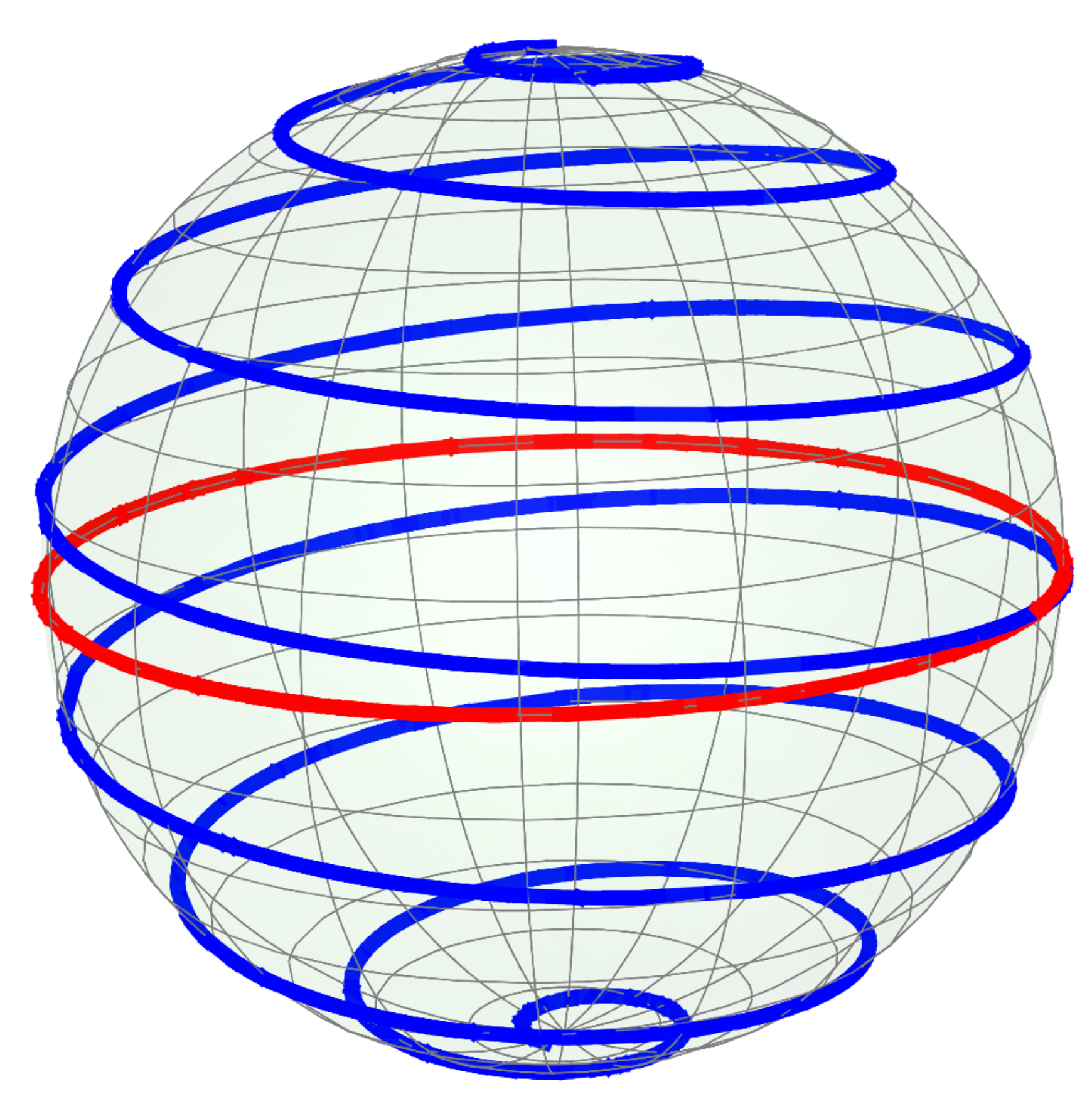}
\caption{Spherical curve with $\mathcal K(z)=-\sqrt{a-z^2}$, $0<a<1$.}
\label{zSqrtaminusz2}
\end{center}
\end{figure}

\subsection{Case $\mu >1$}
Finally we wish to study the spherical curves whose curvature is given by
\begin{equation}
\label{explicit2}
\kappa(z)=\frac{a z}{\sqrt{1-a z^2}}, \, a>1.
\end{equation}
We observe that corresponds in \eqref{k mu delta} with the elections  $\mu=\sqrt a$ and $\delta =1/\sqrt a$.
The trivial solution is given by the equator $z=0$. Otherwise, we make use of the method proposed in Theorem~\ref{quadratures} or Remark \ref{c}, considering the spherical angular momentum $$\mathcal K (z)=-\sqrt{1-a z^2}.$$
In this way, we obtain:
$$
s=s(z)=\int \frac{dz}{\sqrt{1-z^2-(-\sqrt{1-a z^2})^2}}=\int \frac{dz}{\sqrt{a-1}\,|z|}=\frac{\log |z|}{\sqrt{a-1}},
$$
and so $|z(s)|=e^{\sqrt{a-1} \, s}>0$. 

If we write $a=\cosh^2 \delta$, $\delta >0$, then $z(s)=\pm e^{\sinh \delta \, s}$ and
$$
\begin{array}{c}
\lambda (s)= \int \frac{\sqrt{1-\cosh^2 \delta \, e^{2 \sinh \delta \, s}}}{1-e^{2\sinh \delta \, s}}ds= 
\\ \\
\!=\! -\dfrac{\arctanh \left(\sqrt{1\!-\!\cosh^2 \delta \, e^{2 \sinh \delta \, s}}\right)}{\sinh \delta} \!+\! \arctan \left(\dfrac{\sqrt{1\!-\!\cosh^2 \delta \, e^{2 \sinh \delta \, s}}}{\sinh \delta} \right),
\end{array}
$$
where $s<-\log \cosh \delta / \sinh \delta $.

Using \eqref{explicit2}, we get the intrinsic
equation
$$
|\kappa (s)|= 
\dfrac{\cosh^2 \delta \, e^{ \sinh \delta \, s} }{\sqrt{1-\cosh^2 \delta \, e^{2 \sinh \delta \, s}}},
\ s<-\log \cosh \delta / \sinh \delta,
$$
of this family of spherical curves of loxodromic-type (see Figure \ref{zSqrtaz2plus1}) characterized by the geometric angular momentum $\mathcal K (z)=-\sqrt{1-\cosh ^2 \delta z^2}$.
\begin{figure}[h!]
\begin{center}
\includegraphics[height=2.9cm]{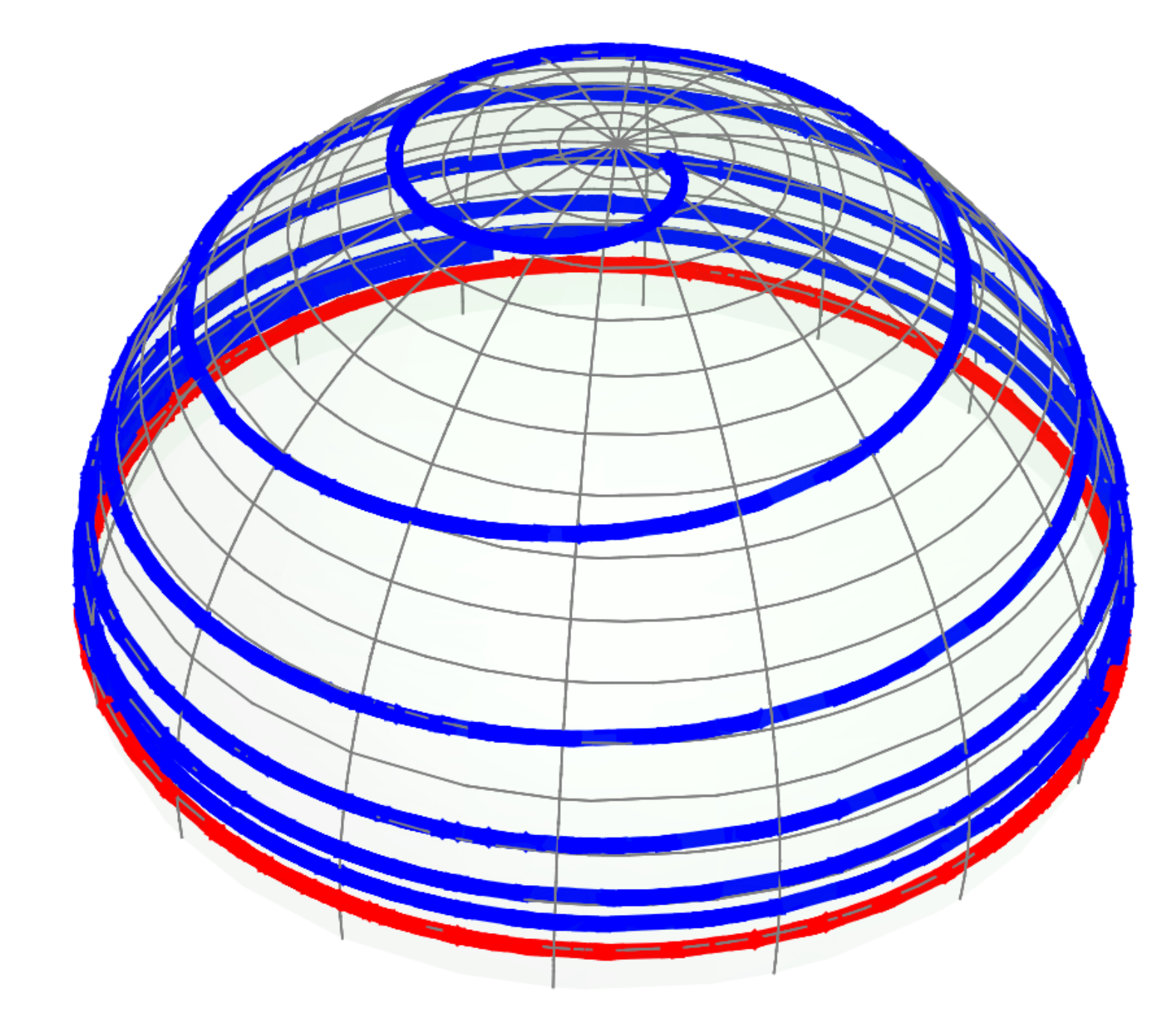}
\caption{Spherical curve with $\mathcal K(z)=-\sqrt{1-a z^2}$, $a>1$.}
\label{zSqrtaz2plus1}
\end{center}
\end{figure}

\section{Spherical catenaries}
In this section we are interested in critical points of the functional
$$ \mathcal F (\xi):=\int_\xi \sqrt{\kappa} \, ds. $$
The above functional $\int_\xi \kappa^{1/2}\, ds$ was considered in Section 5 of \cite{AGM03} acting on the space of convex ($\kappa >0$) closed curves in $\s^2$, motivated by the study of ($r=1/2$)-generalized elastic curves in $\s^2$. We aim to connect the critical points of $\mathcal F$ with the classical catenaries curves in $\s^2$. 

Taking into account Remark \ref{rem:functional}, we are devoted to study the spherical curves with curvature 
\begin{equation}
\label{k mu delta cate}
\kappa (z) = \frac{\delta^2}{4z^2}, \,  \delta \neq 0. 
\end{equation}

The spherical catenaries are the equilibrium lines of an inelastic flexible homogeneous infinitely thin massive wire included in a sphere, placed in a uniform gravitational field. Like any catenaries, their centres of gravity have the minimal altitude among all the curves with given length passing by two given points. They were studied by Bobillier in 1829 and by Gudermann in 1846 (cf.\ \cite{F93}). See Figure \ref{catenaries}.
\begin{figure}[h!]
\begin{center}
\includegraphics[height=3cm]{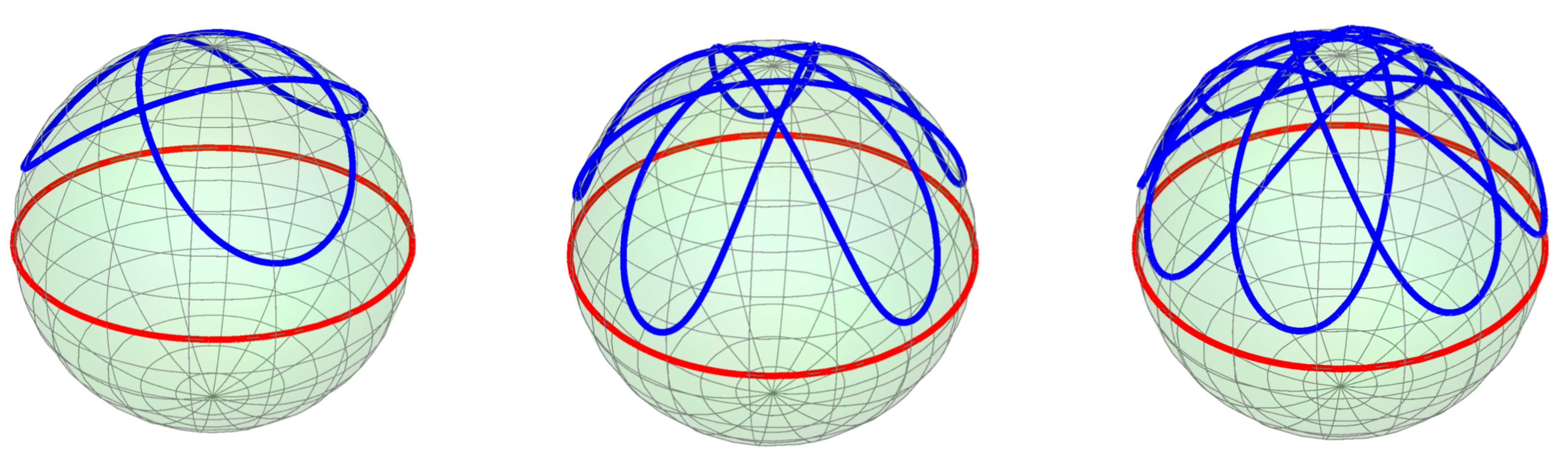}
\caption{Closed spherical catenaries.}
\label{catenaries}
\end{center}
\end{figure}

Using cylindrical coordinates $(r,\theta,z)$ in $\R^3$, they can be described analytically by the following first integral of the corresponding ordinary differential equation:
\begin{equation}\label{catenary}
(z-z_0) \, r^2 \, \frac{d\theta}{ds}= {\rm constant}.
\end{equation}

We study in this section spherical curves satisfying the condition
\begin{equation}\label{k=a/z2}
\kappa(z)=a /z^2, \  a>0.
\end{equation}
So they correspond in \eqref{k mu delta cate} to the election  $\delta =2\sqrt a$.
For any $a>0$, it is easy to prove that there exists an unique angle $\varphi_0 \in (0,\pi/2)$ such that $a=\tan\varphi_0 \sin^2 \varphi_0$. Thus the parallel $z=\sin \varphi_0$ is a constant solution to \eqref{k=a/z2}.

We now apply Theorem~\ref{quadratures} and Remark \ref{c} considering \eqref{k=a/z2} and
$$
\mathcal K (z)=-\frac{a}{z}, \ a>0.
$$
Then we have:
\begin{equation*}
s=s(z)=\int \frac{z\,dz}{\sqrt{z^2(1-z^2)-a^2}},
\end{equation*}
which implies that $a<1/2$ and $1-\sqrt{1-4a^2}<2z^2<1+\sqrt{1-4a^2}$, and it is not difficult to get:  
\begin{equation}\label{z_catenary}
z(s)=\sqrt{\frac{1+\sqrt{1-4a^2}\sin 2s}{2}}.
\end{equation}
In addition, we have that
\begin{equation}
\label{eq catenaries}
d\lambda=\frac{a}{z(1-z^2)}\,ds
\end{equation}

Looking at \eqref{catenary}, taking into account that $r^2+z^2=1$ and $\theta = \lambda$, we deduce from \eqref{eq catenaries} that we get a spherical catenary (with $z_0=0$ and constant $a\in (0,1/2)$). 
However, the explicit computation of $\lambda$ in terms of the arc parameter $s$ requires elliptic integrals of the first and third kind.
As a summary, we have proved the following uniqueness result.
\begin{corollary}\label{cor:catenary}
The spherical catenaries \eqref{eq catenaries} are the only spherical curves (up to rotations around the $z$-axis) with spherical angular momentum $\mathcal K(z)=-a/z$ (and curvature $\kappa (z)=a/z^2$), $0<a<1/2$.
\end{corollary}
Combining \eqref{k=a/z2} and \eqref{z_catenary} we have that the intrinsic equation of the spherical catenaries is given by
$$ \kappa (s)=\frac{2a}{1+\sqrt{1-4a^2}\sin 2s}, \ 0<a<1/2. $$

\section{New and classical spherical curves}
The purpose of this section is to find out new curves $\xi=(x,y,z)$ in $\s^2$, expressed in terms of elementary functions or in terms of Jacobi elliptic functions, prescribing their curvature as a function of the distance from the equator in such a way we can avoid the difficulties described in Remark~\ref{difficulties}.
In addition, we provide uniqueness results for some well known spherical curves in terms of the spherical angular momentum introduced in Section \ref{Sect2}.

\subsection{Spherical curves such that $\kappa(\varphi)=p \cos 2 \varphi/ \cos \varphi, \, 0<p<1$}
The purpose of this section is to find out new curves in $\s^2$ expressed in terms of Jacobi elliptic functions prescribing in a suitable way their curvature in terms of their latitude. Concretely, we aim to study the spherical curves whose curvature is given by
\begin{equation}
\label{Jacobian k}
\kappa(\varphi)=\frac{p \cos 2 \varphi}{\cos \varphi}, \ 0<p<1.
\end{equation}
Recalling that $z=\sin \varphi $, \eqref{Jacobian k} is equivalent to
\begin{equation}
\label{Jacobian k z}
\kappa(z)=\frac{p (1-2z^2)}{\sqrt {1-z^2}}, \ 0<p<1.
\end{equation}
We follow the strategy proposed by Theorem~\ref{quadratures} or Remark \ref{c}, considering the spherical angular momentum 
$$\mathcal K (z)=p\,z \sqrt {1-z^2}.$$ Then we get:
\begin{equation*}
s=\int \frac{dz}{\sqrt{(1-z^2)(1-p^2 z^2)}}=\int \frac{d\varphi}{\sqrt{1-p^2 \sin^2 \varphi}}=
F(\varphi,p)=F(\arcsin z, p),
\end{equation*}
where $F(\cdot,p)$ denotes the elliptic integral of first class of modulus $p$ (see e.g.\ \cite{BF71}).
Hence $\varphi(s)= \am (s,p)$ and $z(s)=\sn (s,p)$, 
where $\am(\cdot,p)$ is the Jacobi amplitude and $\sn(\cdot,p)$ is the Jacobi sine amplitude of modulus $p$ (see e.g.\ \cite{BF71}). In addition: 
\begin{equation*}
\lambda (s)=-p \int \frac{\sn (s,p)}{\cn (s,p)}ds
\end{equation*}
where $\cn(\cdot,p)$ is the Jacobi cosine amplitude of modulus $p$. Using formula 316.01 of \cite{BF71}, we finally arrive at the following expression for the longitude:
$$\lambda (s)=-\frac{p}{2p'}\log \left(  \frac{\dn(s,p)+p'}{\dn(s,p)-p'} \right),$$ 
where  $\dn(\cdot,p)$ is the Jacobi delta amplitude of modulus $p$ and $p'=\sqrt{1-p^2}$ is the complementary modulus.

Using \eqref{Jacobian k z} and that $z(s)=\sn (s,p)$, joint to formula 124.02 of \cite{BF71}, we get the intrinsic equation 
$$
\kappa (s) = p\left( 2 \cn(s,p)-1/\cn(s,p) \right), \ 0<p<1,
$$
of the only spherical curves (up to rotations around $z$-axis) with spherical angular momentum  $\mathcal K (z)=p\,z \sqrt {1-z^2}$ or, equivalently, $\mathcal K (\varphi)=(p/2) \sin 2 \varphi$,  $ 0<p<1$. 

These curves are embedded and closed since $\xi (s+4 K(p))=\xi (s)$, where $K(p)$ is the complete elliptic integral of first class of modulus $p$ (see Figure~\ref{Jacobian curves}).
\begin{figure}[h!]
\begin{center}
\includegraphics[height=2.9cm]{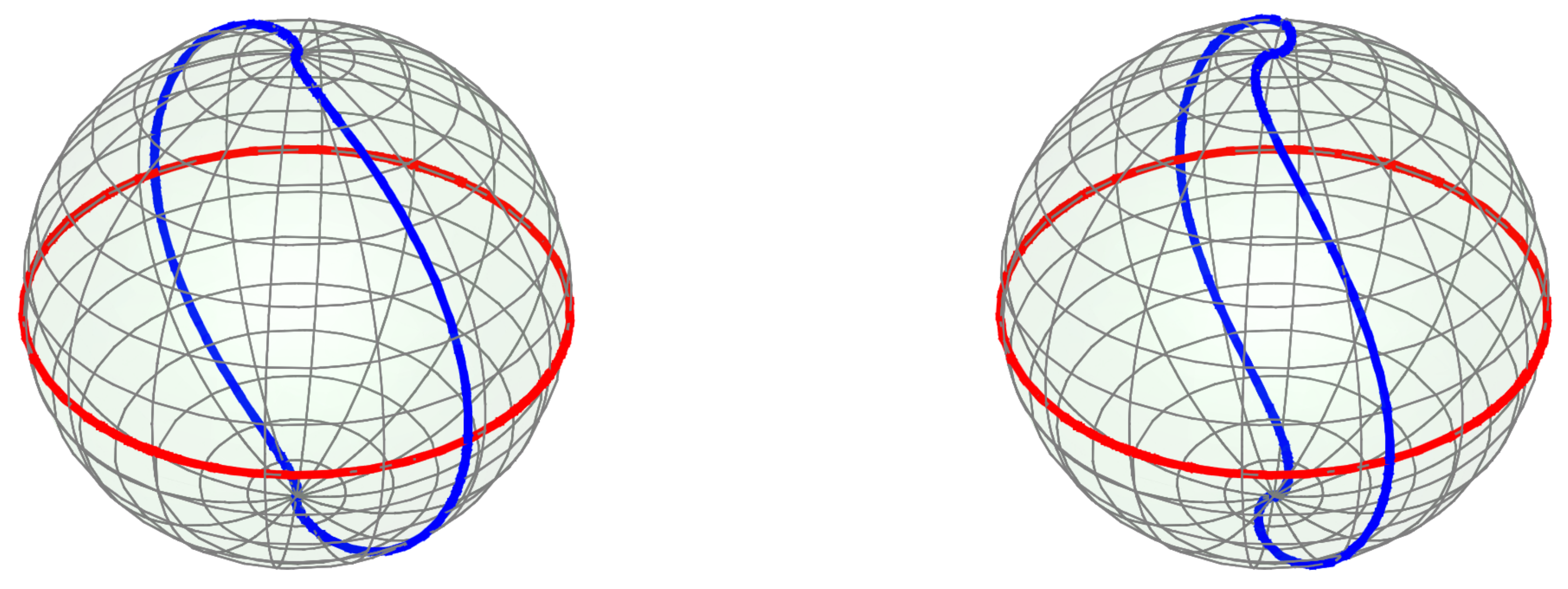}\\
\includegraphics[height=3cm]{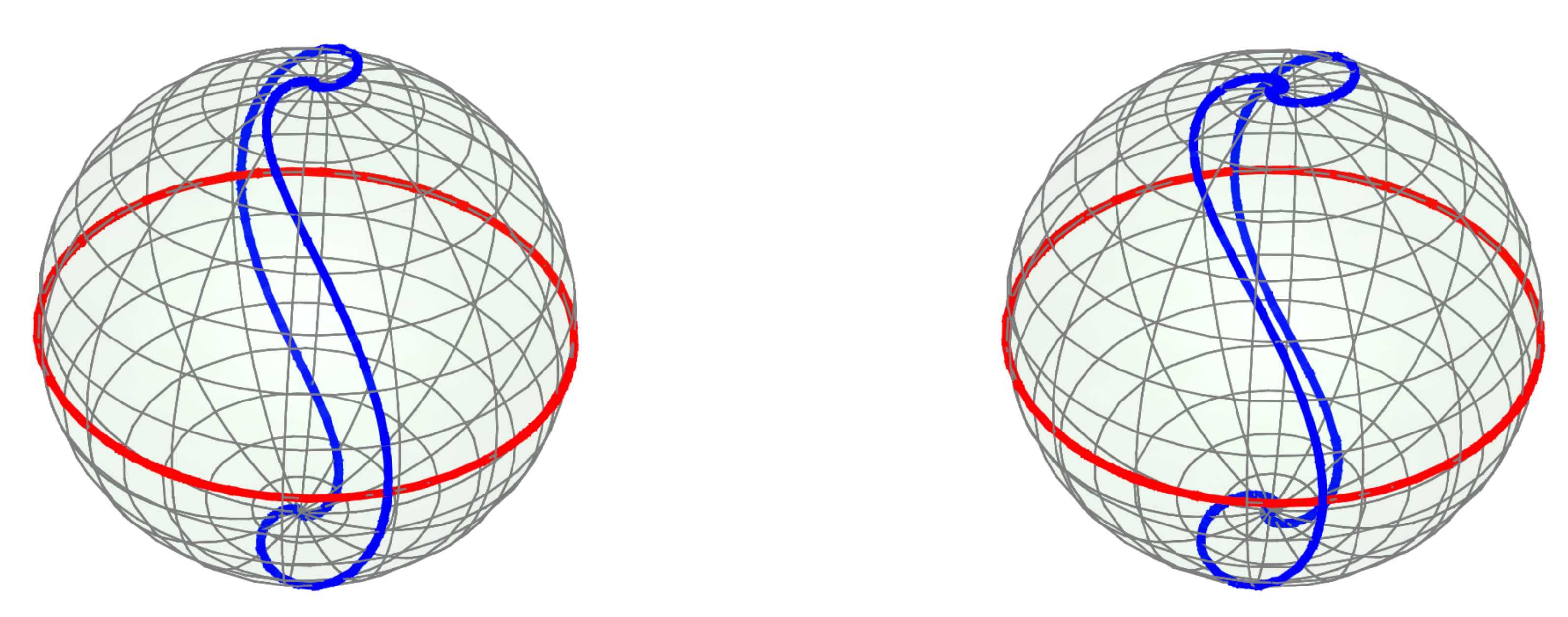}
\caption{Spherical curves with  $\mathcal K (z)\!=\!p\,z \sqrt {1\!-\!z^2}$, $0\!<\!p\!<\!1$.}
\label{Jacobian curves}
\end{center}
\end{figure}

\subsection{Viviani's curve and spherical Archimedean spirals}
The Viviani's curve is the intersection between a sphere of radius $R$ and a cylinder of revolution with diameter $R$ such that a generatrix passes by the centre of the sphere; so this curve is at the same time spherical and cylindrical.
We can obtain a Viviani's curve by sticking the tip of a compass inside a cylinder of revolution and tracing on this cylinder a ``circle'' with radius equal to the cylinder diameter. It was studied by Vincenzo Viviani in 1692 (cf.\ \cite{F93}). In geographical coordinates of $\s^2$, the Viviani's curve can be simply described as $\varphi = \lambda$ (see Figure \ref{Viviani}).

We study in this section spherical curves satisfying the condition
\begin{equation}\label{k=_Viviani}
\kappa(z)=\frac{z(3-z^2)}{(2-z^2)^{3/2}},
\end{equation}
applying Theorem~\ref{quadratures} and Remark \ref{c}, considering the spherical angular momentum
$$
\mathcal K (z)= \frac{z^2-1}{\sqrt{2-z^2}}.
$$
Then we have:
\begin{equation}
\label{s Viv}
s=s(z)=\int \sqrt{\frac{2-z^2}{1-z^2}}\, dz=E(\arcsin z,1/2)
\end{equation}
which involves an elliptic integral $E$ of second kind (see e.g.\ \cite{BF71}).
In addition, we get:
\begin{equation}
\label{long Viv}
d\lambda=\frac{ds}{\sqrt{2-z^2}}
\end{equation}
Using that $z=\sin \varphi$, \eqref{s Viv} and \eqref{long Viv}, we get easily that $d\lambda=d\varphi$.
Hence we have proved the following characterization of the Viviani's curve.
\begin{corollary}\label{cor:Viviani}
The Viviani's curve $\varphi = \lambda$ is the only spherical curve (up to rotations around the $z$-axis) with spherical angular momentum $\mathcal K(\varphi)=-\cos^2 \varphi/\sqrt{1+\cos^2 \varphi}$. 
\end{corollary}
\begin{figure}[h!]
\begin{center}
\includegraphics[height=3cm]{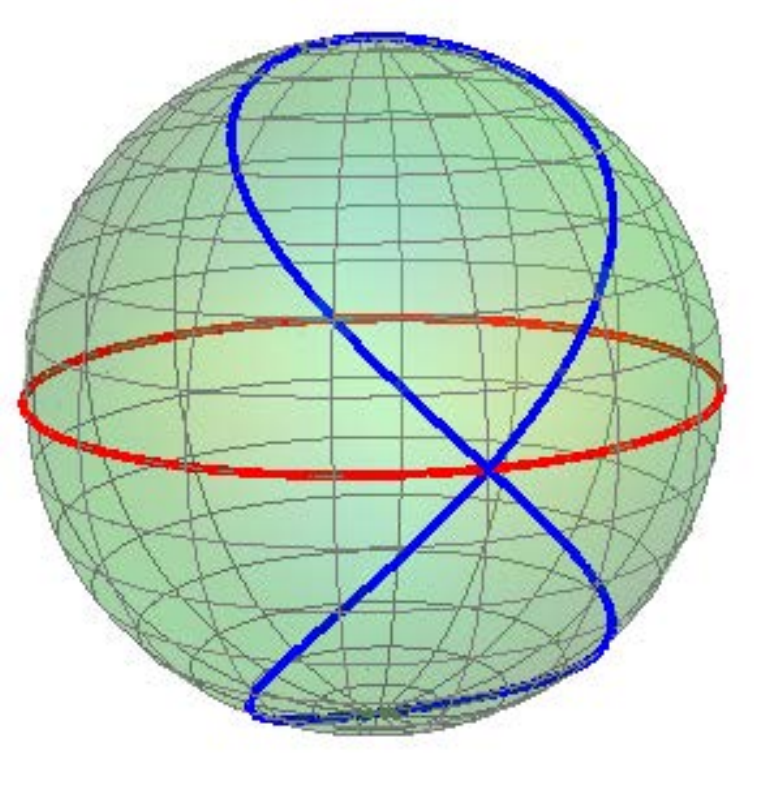}
\caption{Viviani's curve.} 
\label{Viviani}
\end{center}
\end{figure}

The spherical Archimedean spiral curves are natural generalizations of the Viviani's curve, since they are described in geographical coordinates by $\varphi = n \lambda$, $n>0$. A spherical Archimedean spiral is algebraic if and only if $n\in \Q$. They were studied by Guido Grandi in 1728, also called \textit{clelias}. They are the loci of a point $P$ on a meridian of a sphere rotating at constant speed $\omega$ around the polar axis, the point $P$ also moving at constant speed $n\,\omega$ along this meridian (see Figure \ref{spiral}). 
Therefore, physically, we obtain a clelia when peeling an orange or when rewinding regularly a spherical wool ball.
\begin{figure}[h!]
\begin{center}
\includegraphics[height=3cm]{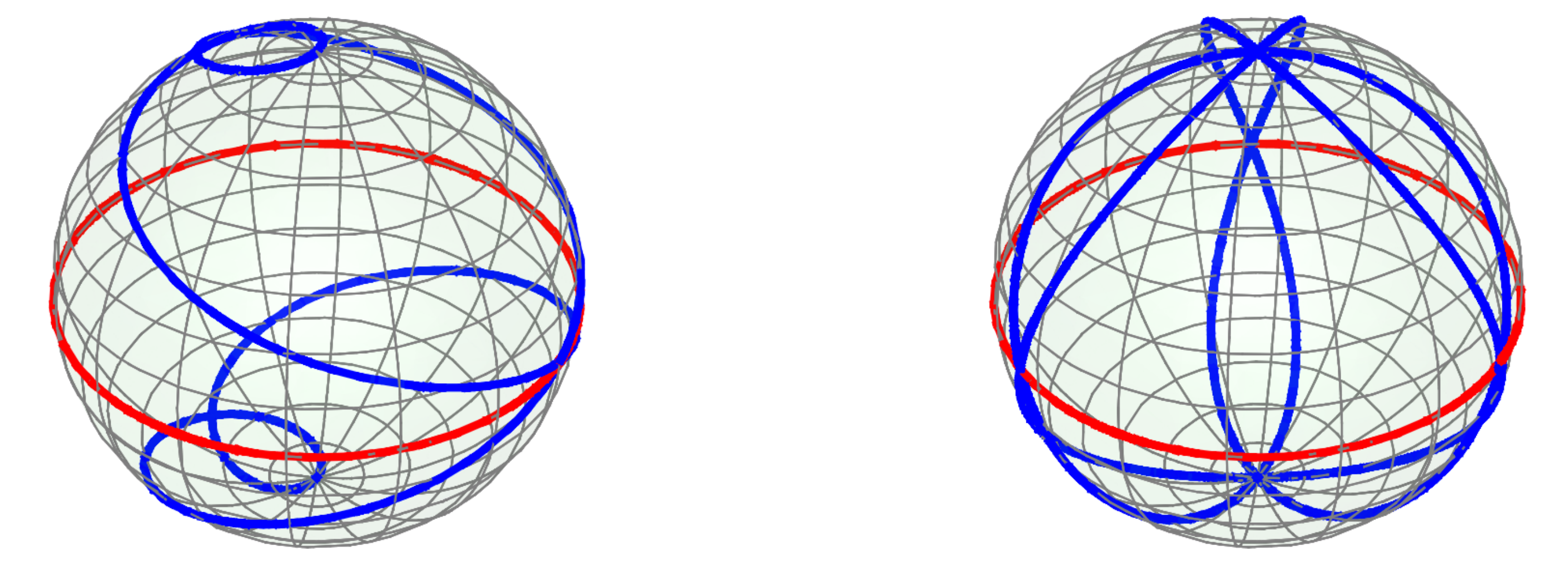}
\caption{Spherical Archimedean spiral curves: $\varphi = n \lambda$, $n=1/3$ and $n=3$.}
\label{spiral}
\end{center}
\end{figure}

A similar argument to the used in the preceding section, considering now
\begin{equation}\label{k=_spirals}
\kappa(z)=\frac{z(2n^2+1-z^2)}{(n^2+1-z^2)^{3/2}}
\end{equation}
give us the following uniqueness result, whose proof we will omit.
\begin{corollary}\label{cor:spirals}
The spherical spiral curves $\varphi = n \lambda$, $n>0$, are the only spherical curves (up to rotations around the $z$-axis) with spherical angular momentum $\mathcal K(\varphi)=-\cos^2 \varphi/\sqrt{n^2+\cos^2 \varphi}$.
\end{corollary}

\subsection*{Acknowledgements}
The authors were partially supported by State Research Agency and European Regional Development Fund via the grant no. PID\ 2020-117868GB-I00 supported by MCIN/AEI/10.13039/ 501100011033.
Second named author also by MICINN/FEDER project PGC2018-097046-B-I00 and Fundaci\'on S\'eneca project 19901/GERM/15, Spain. 
Third named author was supported by Predoc grant no. FPU16/03096 supported by MCIN/AEI/10.13039/501100011033.
%

We would also like to thank \'Alvaro P\'ampano for helpful discussion about the topic, specially for comments concerning Remark \ref{rem:functional}.

\end{document}